\documentclass[12pt,reqno]{amsart}

\usepackage{amssymb}

\usepackage{mathrsfs}

\usepackage{bbm}

\DeclareMathAlphabet{\mathpzc}{OT1}{pzc}{m}{it}

\usepackage[all]{xy}

\entrymodifiers={+!!<0pt,\fontdimen22\textfont2>}

\def\BZ{\mathbb{Z}}

\def\Add{\operatorname{Add}}
\def\adots{\mathinner{\mkern1mu\raise1.0pt\vbox{\kern7.0pt\hbox{.}}\mkern2mu\raise4.0pt\hbox{.}\mkern2mu\raise7.0pt\hbox{.}\mkern1mu}}

\def\C{\operatorname{C}}

\def\H{\operatorname{H}}

\def\Hom{\operatorname{Hom}}
\def\HProj{\mbox{\rm $H$-Proj}}
\def\HProjsmall{\mbox{\tiny $H$-Proj}}

\def\Ind{\operatorname{Ind}}

\def\K{\operatorname{K}}

\def\Tate{\operatorname{Tate}}

\def\Mod{\operatorname{Mod}}

\def\Res{\operatorname{Res}}

\def\StMod{\operatorname{StMod}}

\def\Z{\operatorname{Z}}

\newtheorem{Lemma}{Lemma}[section]

\newtheorem{Theorem}[Lemma]{Theorem}
\newtheorem{Proposition}[Lemma]{Proposition}
\newtheorem{Corollary}[Lemma]{Corollary}

\theoremstyle{definition}
\newtheorem{Definition}[Lemma]{Definition}

\newtheorem{Remark}[Lemma]{Remark}

\begin{document}

\setlength{\parindent}{0pt}
\setlength{\parskip}{7pt}

\title[Relative compact generation]{Some relative stable categories
  are compactly generated}

\author{Matthew Grime}

\author{Peter J\o rgensen}
\address{J\o rgensen: School of Mathematics and Statistics,
Newcastle University, Newcastle upon Tyne NE1 7RU,
United Kingdom}
\email{peter.jorgensen@ncl.ac.uk}
\urladdr{http://www.staff.ncl.ac.uk/peter.jorgensen}





\begin{abstract} 
  
  Let $G$ be a finite group.  The stable module category of $G$ has
  been applied extensively in group representation theory.  In
  particular, it has been used to great effect that it is a
  triangulated category which is compactly generated.
  
  Let $H$ be a subgroup of $G$.  It is possible to define a stable
  mo\-du\-le category of $G$ relative to $H$.  It too is a triangulated
  category, but no non-trivial examples have been known where this
  relative stable category was compactly generated.
  
  We show here that the relative stable category is compactly
  generated if the group algebra of $H$ has finite representation
  type.  In characteristic $p$, this is equivalent to the Sylow
  $p$-subgroups of $H$ being cyclic.
 
\end{abstract}

\maketitle


The study of localizations of triangulated categories has a rich and
varied heritage arising from the work of Adams, Bousfield, Brown,
Thomason, and others. Neeman further developed these theories, and
showed that to bring the full power of such arguments to bear one
needs a compactly generated category.

Localization techniques were brought to the attention of the
representation theory world in Rickard's \cite{R}, where they were
applied to the stable module category which is easily shown to be
compactly ge\-ne\-ra\-ted.

The stable module category is not the only triangulated quotient of
the module category that one meets in representation theory. In
\cite{CPW} Carlson, Peng, and Wheeler note that one can adapt
Rickard's work to relative stable module categories. However, not much
is known about the structure of these categories. In particular no
non-trivial examples have been given which are known to be compactly
generated.

In this note we prove the following theorem on the relative stable
mo\-du\-le category $\StMod_H(kG)$:

\label{pag:1}
{\bf Theorem.}  {\em Let $k$ be an algebraic closure of $\BZ/p$.  Let
  $G$ be a finite group, $H$ a subgroup of $G$.  If $kH$ has finite
  representation type, then $\StMod_H(kG)$ is compactly generated.}

We will keep the assumptions on $k$, $G$, and $H$ for the rest of the
paper.  All modules will be left-modules.  Recall that $kH$ has finite
representation type precisely if the Sylow $p$-subgroups of $H$ are
cyclic, see \cite[thm.\ VI.3.3]{ARS}.  Let us remind the reader of the
construction of $\StMod_H(kG)$.

We will denote the class of $H$-projective $kG$-modules by $\HProj$;
this is the class of all summands of modules induced up from $kH$. It
is an {\em additive} subcategory of $\Mod(kG)$. We use $\HProj$ to
define the triangulated categories $\StMod_H(kG)$ and $\K(\HProj)$.

The relative stable module category $\StMod_H(kG)$ is $\Mod(kG)$ modulo
the morphisms that factor through objects of $\HProj$. The category
$\K(\HProj)$ is the homotopy category of complexes of objects of
$\HProj$.

By $\Tate_H(kG)$ we denote the collection of complexes $Q$ of
$\HProj$-modules for which the restriction $\Res_H^G(Q)$ to $H$ is
split exact.  We may think of $\Tate_H(kG)$ either as a triangulated
subcategory of $\K(\HProj)$, or as a full subcategory of $\C(\HProj)$,
the category of complexes of objects of $\HProj$ and chain maps.

\begin{Remark}
\label{rmk:Tate}
If $X$ is in $\Tate_H(kG)$ then $X$ is exact and splits into
short exact sequences
\[
  0 \rightarrow \Z^n(X) \rightarrow X^n \rightarrow \Z^{n+1}(X) \rightarrow 0
\]
which become split exact upon restriction to $H$.  In particular, it
is easy to show that $\Z^n(X) \rightarrow X^n$ is an
$\HProj$-preenvelope and $X^n \rightarrow \Z^{n+1}(X)$ is an
$\HProj$-precover.
\end{Remark}

Several variants of the following result are well known, see for
instance \cite[thm.\ 2.3]{BG} and \cite[thm.\ 9.6.4]{HPS}.

\begin{Proposition}
\label{pro:equiv}
View $\Tate_H(kG)$ as a full subcategory of $\K(\HProj)$.
There is a triangulated equivalence of categories
\[
  \Tate_H(kG) \simeq \StMod_H(kG) 
\]
given by $X \mapsto \Z^0(X)$.
\end{Proposition}

\begin{proof}
We can clearly view $\Z^0$ as a functor $\C(\HProj) \rightarrow
\Mod(kG)$.  Viewing $\Tate_H(kG)$ as a full subcategory of
$\C(\HProj)$, we hence have a functor
\begin{equation}
\label{equ:a}
  \Z^0 : \Tate_H(kG) \rightarrow \Mod(kG).
\end{equation}

Let $X \rightarrow Y$ be a chain map of complexes from $\Tate_H(kG)$.
There is an induced homomorphism $\Z^0(X) \rightarrow \Z^0(Y)$.  The
chain map is null homotopic if and only if the induced homomorphism
factors through a module from $\HProj$, that is, if and only if the
induced homomorphism becomes $0$ in $\StMod_H(kG)$.  This holds by a
lifting argument using Remark \ref{rmk:Tate}; cf.\ \cite[proof of
lem.\ 2.2]{BG}.

Hence the functor from equation \eqref{equ:a} induces a faithful
functor
\begin{equation}
\label{equ:b}
  \Z^0 : \Tate_H(kG) \rightarrow \StMod_H(kG)
\end{equation}
where $\Tate_H(kG)$ is now viewed as a full subcategory of
$\K(\HProj)$.

Observe that $X$ can be viewed as a relative Tate resolution of
$\Z^0(X)$.  Hence the functor from \eqref{equ:b} is also full, since
any homomorphism of modules can be lifted to the relative Tate
resolutions; this is again a lifting argument using Remark
\ref{rmk:Tate}.

To conclude that the functor from \eqref{equ:b} is an equivalence of
categories, all that is needed is to see that it is essentially
surjective.  But each $kG$-module $m$ has a relative Tate resolution
$X$, so indeed, $m \cong \Z^0(X)$ for some $X$.  Note that we can
construct such an $X$ by splicing a left-$\HProj$-resolution and a
right-$\HProj$-resolution of $m$.  These resolutions become split
exact upon restriction to $H$ because this is true for
$\HProj$-precovers and -preenvelopes.
\end{proof}

\begin{Definition}
If $kH$ has finite representation type, then $y$ will be the
direct sum of its indecomposable finitely generated modules, and
$x = \Ind_H^G(y)$ the induced module over $kG$.  
\end{Definition}

\begin{Remark}
\label{rmk:x}
In the case of the definition, note that $\HProj = \Add(x)$.  Note
also that $x$ can be viewed as a complex concentrated in degree zero.
As such, it is in $\K(\HProj)$.
\end{Remark}

\begin{Lemma}
\label{lem:perp}
We have
{\rm 
\begin{align*}
  \lefteqn{\Tate_H(kG) = x^{\perp}} & \\
  & \;\;\;
    = \{\, Q \in \K(\HProj) \,|\, 
           \Hom_{\K(\HProjsmall)}(\Sigma^n x,Q) = 0 \mbox{ for each } n \,\}
\end{align*}
}
\!\!in $\K(\HProj)$.
\end{Lemma}

\begin{proof}
Let $Q$ be in $\K(\HProj)$.  Then
\begin{align*}
  \Hom_{\K(\HProjsmall)}(\Sigma^n x,Q)
  & = \Hom_{\K(kG)}(\Sigma^n \Ind_H^G(y),Q) \\
  & \cong \Hom_{\K(kH)}(\Sigma^n y,\Res_H^G(Q)) \\
  & = (*)
\end{align*}
by adjointness, since $\Ind_H^G(y) = kG \otimes_{kH} y$ while
$\Res_H^G$ restricts $kG$-modules to $kH$-modules.  If $(*)$ is $0$
then so is
\[
  \Hom_{\K(kH)}(\Sigma^n m,\Res_H^G(Q))
\]
for each $m$ in $\Mod(kH)$, since $\Mod(kH)$ equals $\Add(y)$ by
\cite[cor.\ 4.8]{A} because $kH$ has finite representation type.  But
if this expression is $0$ for each $m$ and each $n$, then
$\Res_H^G(Q)$ is null homotopic by an easy argument; that is, $Q$ is
in $\Tate_H(kG)$.
\end{proof}

\begin{Proposition}
\label{pro:comp}
If $kH$ has finite representation type, then $\K(\HProj)$ is
compactly generated.
\end{Proposition}

\begin{proof}
Since $k$ is countable, $kG$ has pure global dimension $\leq 1$ by
\cite[thm.\ 11.21]{JL}.  The finite dimensional algebra $kG$ is
certainly coherent, and $x$ is a finitely generated $kG$-module.

By \cite[sec.\ 4, (1)]{HJ}, the category $\K(\Add x)$ is compactly
generated.  But this category is $\K(\HProj)$ by Remark \ref{rmk:x}.
\end{proof}

\begin{Corollary}
\label{cor:comp}
  If $kH$ has finite representation type, then $\Tate_H(kG)$ is
  compactly generated.
\end{Corollary}

\begin{proof}
The category $\K(\HProj)$ is compactly generated by Proposition
\ref{pro:comp}, and $\Tate_H(kG) = x^{\perp}$ by Lemma
\ref{lem:perp}.

But $x$ is a compact object of $\K(\HProj)$, as follows for instance
from the formula
\[
  \Hom_{\K(\HProjsmall)}(x,-) \simeq \H^0 \Hom_{kG}(x,-)
\]
since $x$ is finitely generated over $kG$.

So $\Tate_H(kG)$ is the right perpendicular category of a compact
object, so it is compactly generated by \cite[prop.\ 1.7(1)]{IK}.
\end{proof}

Finally, the theorem from page \pageref{pag:1} follows.

\begin{Theorem}
\label{thm:comp}
  If $kH$ has finite representation type, then $\StMod_H(kG)$ is
  compactly generated.
\end{Theorem}

\begin{proof}
Combine Proposition \ref{pro:equiv} with Corollary \ref{cor:comp}.
\end{proof}

\begin{Remark}
It is not clear that our methods can be used to compute a set of
compact generators of $\StMod_H(kG)$.

To do so, we would need to find a set of compact generators of the
category $\Tate_H(kG)$ and then use the equivalence $\Z^0$.  By
unravelling the proof of \cite[prop.\ 1.7(1)]{IK}, it can be seen that
the compact generators of $\Tate_H(kG)$ would come by taking a set of
compact generators of $\K(\HProj)$ and applying the left adjoint to
the inclusion of $\Tate_H(kG)$ into $\K(\HProj)$.  This left adjoint
is constructed by Neeman in \cite{N}, but the construction is infinite
and does not obviously lend itself to concrete computations.

It would be interesting to find a procedure whereby a set of compact
generators of $\StMod_H(kG)$ could be computed.
\end{Remark}

\bigskip
\noindent
{\bf Acknowledgements.}
The first author is (partly) supported by the Heilbronn Institute for
Mathematical Research.


\begin{thebibliography}{19}

\bibitem{A} M.\ Auslander, {\it Representation theory of Artin
    algebras II}, Comm.\ Algebra {\bf 1} (1974), 269--310.
  
\bibitem{ARS}  M.\ Auslander, I.\ Reiten, and S.\ Smal\o,
  ``Representation theory of Artin algebras'', Cambridge Stud.\ Adv.\
  Math., Vol.\ 36, Cambridge University Press, Cambridge, 1995.  First
  paperback edition with corrections, 1997.

\bibitem{BG} D.\ J.\ Benson and J.\ P.\ C.\ Greenlees, {\it
    Localization and duality in topology and modular representation
    theory}, J.\ Pure Appl.\ Algebra {\bf 212} (2008), 1716--1743.
  
\bibitem{CPW} J.\ F.\ Carlson, C.\ Peng, and W.\ W.\ Wheeler, {\it
    Transfer maps and virtual projectivity}, J.\ Algebra {\bf 204}
  (1998), 286--311.
  
\bibitem{HJ} H.\ Holm and P.\ J\o rgensen, {\it Compactly generated
    homotopy categories}, Ho\-mo\-lo\-gy, Homotopy Appl.\ {\bf 9}
  (2007), 257--274.
  
\bibitem{HPS} M.\ Hovey, J.\ H.\ Palmieri, and N.\ P.\ Strickland,
  {\it Axiomatic stable homotopy theory}, Mem.\ Amer.\ Math.\ Soc.\ 
  {\bf 128} (1997), no.\ 610.

\bibitem{IK} S.\ Iyengar and H.\ Krause, {\it Acyclicity versus total
    acyclicity for complexes over noetherian rings}, Doc.\ Math.\ {\bf
    11} (2006), 207--240.

\bibitem{JL} C.\ U.\ Jensen and H.\ Lenzing, ``Model theoretic
  algebra'', Algebra Logic Appl., Vol.\ 2, Gordon and Breach, New
  York, 1989.
  
\bibitem{N} A.\ Neeman, {\it The connection between the $K$-theory
    localization theorem of Thomason, Trobaugh and Yao and the
    smashing subcategories of Bousfield and Ravenel}, Ann.\ Sci.\ 
  \'{E}cole Norm.\ Sup.\ (4) {\bf 25} (1992), 547--566.
  
\bibitem{R} J.\ Rickard, {\it Idempotent modules in the stable
    category}, J.\ London Math.\ Soc.\ (2) {\bf 56} (1997), 149--170.

\end{thebibliography}
\end{document}